\documentclass[reqno,11pt,english]{amsart}

\usepackage{amsfonts}
\usepackage{graphicx}
\usepackage{amsmath}
\usepackage{amssymb}
\usepackage{amsthm}
\usepackage{color}
\usepackage{cancel}
\usepackage{enumerate}
\usepackage{tikz}
\tikzset{vertex/.style={circle,fill,draw,scale=.2}}
\usepackage[all, cmtip]{xy} 
\newtheorem{theorem}{Theorem}

\newtheorem{corollary}[theorem]{Corollary}

\newtheorem{definition}[theorem]{Definition}

\newtheorem{lemma}[theorem]{Lemma}

\newtheorem{proposition}[theorem]{Proposition}

\usepackage[pagebackref,colorlinks]{hyperref}

\usetikzlibrary{decorations.markings}
\tikzstyle{vertex}=[circle, draw, fill=black, inner sep=0pt, minimum size=6pt]

}}}]}

\def\End{\mbox{End\/}}

\newcommand{\gr}{\operatorname{gr}}
\newcommand{\Ma}{\operatorname{\mathbb M}}

\newcommand{\HOM}{\operatorname{HOM}}

\begin{document}
\title[Structure theory of graded regular graded self-injective rings]{Structure theory of graded regular graded self-injective rings and applications}

\author[R. Hazrat]{Roozbeh Hazrat}
\address{Centre for Research in Mathematics, Western Sydney University, AUSTRALIA}
\email{r.hazrat@westernsydney.edu.au}

\author[K.M. Rangaswamy]{Kulumani M. Rangaswamy}
\address{Departament of Mathematics, University of Colorado at Colorado Springs, Colorado-80918, USA}
\email{krangasw@uccs.edu}

\author[A.K. Srivastava]{Ashish K. Srivastava}
\address{Department of Mathematics and Statistics, St. Louis University, St.
Louis, MO-63103, USA} \email{ashish.srivastava@slu.edu}
\dedicatory{To S. K. Jain on his 80th birthday}
\thanks{The first author would like to acknowledge Australian Research Council grant  DP160101481. A part of this work was done at the University of M\"unster, where he was a Humboldt Fellow. The work of the third author is partially supported by a grant from Simons Foundation (grant number 426367).}
\maketitle

\begin{abstract}
In this paper, we develop structure theory for graded regular graded self-injective rings and apply it in the context of Leavitt path algebras. We show that for a finite graph, graded regular graded self-injective Leavitt path algebras are of graded type I and these are precisely graded $\Sigma$-$V$ Leavitt path algebras. 
\end{abstract}

\bigskip

\section{Introduction} 

\bigskip

 This is a semi-expository article illustrating how the type theory to study the structure of self-injective von Neumann regular  rings with identity can be extended to the case of  group graded rings. 
This is accomplished by adapting some of the ideas of Goodearl \cite{G} to the graded situation. As an application, we discuss the type theory of unital Leavitt path algebras over a field. 

Kaplansky developed a classification for Baer rings in \cite{Kap2}. Since (von Neumann) regular right self-injective rings are Baer rings, this classification theory of Baer rings applies to them. Let $R$ be a regular right self-injective ring. Then the ring $R$ is said to be of Type $I$ provided it contains a faithful abelian idempotent. The ring $R$ is said to be of Type $II$ provided $R$ contains a faithful directly finite idempotent but $R$ contains no nonzero abelian idempotents and the ring $R$ is said to be of Type $III$ if it contains no nonzero directly finite idempotents.

A regular right self-injective ring $R$ is said to be of (i) Type $I_{f}$ if $R$ is of Type $I$ and is directly-finite, (ii) Type $I_{\infty }$ if $R$ is of Type $I$
and is purely infinite (that is, $R_R\cong R_R\oplus R_R$), (iii) Type $II_{f}$ if $R$ is of Type $II$ and is directly-finite, (iv) Type $II_{\infty }$ if $R$ is of Type $II$ and is purely infinite.

It is known that \cite[Theorem 10.13]{G}, if $R$ is a regular right self-injective ring, then there is a unique decomposition $R=R_1\times R_2 \times R_3$, where  $R_1$ is of Type $I$,  $R_2$ is of Type $II$, and  $R_3$ is of Type $III$. This is a special case of Kaplansky's decomposition theorem for Baer rings~\cite[Theorem~11]{Kap2} as developed by Goodearl in the setting of regular self-injective rings. 

Our motivation here to develop this theory in the graded setting was the result of \cite{H-2} which shows that all Leavitt path algebras are graded regular and \cite{HR} where the graded self-injective Leavitt path algebras were studied.  Thus one would like to know how these algebras decompose to different \emph{graded} types.


In this paper we first obtain a graded version of the decomposition of a regular ring into different types, by suitably modifying the arguments from \cite{G}. One requires to take care of grading along the way and this forces some subtle changes to the details of the proofs from nongraded version to the graded setting (see the introduction to Section \S2). We then finish the paper by applying these results in the context of Leavitt path algebras. In particular, we provide a characterization for when Leavitt path algebra over a finite graph $E$ is of graded Type I (which immediately shows they are in fact of the Type $I_f$).  

Unless otherwise stated all the rings that we consider possess a multiplicative identity and all the modules are (unitary) right modules. Also, by a self-injective ring, we mean a right self-injective ring.

\section{Structure theory of graded regular graded self-injective rings}\label{sec2}

 The aim of this section is to develop a structure theory for graded von Neumann regular
graded self-injective rings. The class of von Neumann regular rings (regular for short) constitutes a large class  and from the categorical view point they are a natural generalization of the basic building blocks of ring theory, i.e., simple rings. For an associative ring with identity if all modules are free or projective, then the ring is a division ring or a semi-simple ring, respectively. However if all modules are flat then the ring is regular. There is an element-wise definition for such rings; any element has an ``inner inverse'', i.e., for an element $a$, there is $b$ such that $aba=a$. Such rings have very rich structures, and Ken Goodearl has devoted an entire book on them~\cite{G}. In the case of regular rings which are self injective, there is a structure theory developed which shows that the ring can be decomposed uniquely into three types based on the behavior of idempotent elements ~\cite[\S10]{G}.

For the class of graded rings, one can define the notion of graded von Neumann regular rings (graded regular for short) as rings in which any homogeneous element has an inner inverse. There are many interesting class of rings which are not regular but are graded regular. One such class is the Leavitt path algebras~\cite{H-2}. 

In order to carry over the structure theory of regular self injective rings from nongraded setting to the graded case one needs to carefully analyze the behavior of the homogeneous components throughout the proofs, as in the graded setting, suspensions of the modules would come into consideration. A prototype example of regular self injective rings is the ring of column finite matrices over a division ring. This would not readily generalize to the graded setting; the column finite matrices over a graded division rings is not necessarily graded. Throughout this section our rings have identity element. 

We begin with some basics of graded rings and graded modules that we will need throughout this paper. Let $\Gamma$ be an additive abelian group and $R$ a unital  $\Gamma$-graded ring. Namely, $R=\bigoplus_{\gamma \in \Gamma} R_\gamma$, where $R_\gamma$ are additive subgroups of $R$ and $R_\gamma R_\beta \subseteq R_{\gamma+\beta}$, for $\gamma,\beta \in \Gamma$. We call the elements of components $R_\gamma$, $\gamma \in \Gamma$, homogenous elements. We refer the reader to \cite{H,NvO} for the basics on graded rings. Let $M$ and $N$ be graded right $R$-modules.
Consider the hom-group $\operatorname{Hom}_{R}(M,N)$. One can
show that if $M$ is finitely generated or $\Gamma$ is a finite group, then
\[
\operatorname{Hom}_{R}(M,N)=\bigoplus_{\gamma\in\Gamma} \operatorname{Hom}%
(M,N)_{\gamma},
\]
where $\operatorname{Hom}(M,N)_{\gamma}=\{f:M\rightarrow N : f(M_{\alpha
})\subseteq N_{\alpha+\gamma}, \alpha\in\Gamma\}$ (see \cite[\S1.2.3]{H}).
However if $M$ is not finitely generated, then throughout we will work with
the group
\[
\operatorname{HOM}_{R}(M,N):=\bigoplus_{\gamma\in\Gamma} \operatorname{Hom}%
(M,N)_{\gamma},
\]
and consequently the $\Gamma$-graded ring
\begin{equation}\label{hhh}
\operatorname{END}_{R}(M):=\operatorname{HOM}_{R}(M,M).
\end{equation}

If $f\in\operatorname{END}_{R}(M)$ then $f=\sum_{\gamma\in\Gamma} f_{\gamma}$,
where $f_{\gamma}\in\operatorname{Hom}(M,M)_{\gamma}$ and we write
$\operatorname{supp}(f)=\{\gamma\in\Gamma : f_{\gamma}\not =0\}$. This is one point of departure from nongraded setting where we should be working with $\End_R(M)$. 

\label{matrixGrading}Let $R$ be a
$\Gamma$-graded ring and $(\delta_{1},\cdot\cdot\cdot,\delta_{n})$ an
$n$-tuple where $\delta_{i}\in\Gamma$. Then $\Ma_{n}(R)$ is a $\Gamma$-graded
ring and, for each $\lambda\in\Gamma$, its $\lambda$-homogeneous component
consists of $n\times n$ matrices%

\begin{equation}\label{one}
\Ma_{n}(R)(\delta_{1},\cdot\cdot\cdot,\delta_{n})_{\lambda}=\left(
\begin{array}
[c]{cccccc}%
R_{\lambda+\delta_{1}-\delta_{1}} & R_{\lambda+\delta_{2}-\delta_{1}} & \cdot
& \cdot & \cdot & R_{\lambda+\delta_{n}-\delta_{1}}\\
R_{\lambda+\delta_{1}-\delta_{2}} & R_{\lambda+\delta_{2}-\delta_{2}} &  &  &
& R_{\lambda+\delta_{n}-\delta_{2}}\\
&  &  &  &  & \\
&  &  &  &  & \\
&  &  &  &  & \\
R_{\lambda+\delta_{1}-\delta_{n}} & R_{\lambda+\delta_{2}-\delta_{n}} &  &  &
& R_{\lambda+\delta_{n}-\delta_{n}}%
\end{array}
\right).
\end{equation}

This shows that for each homogeneous element $x\in R$,
\begin{equation}\label{two}
\deg(e_{ij}(x))=\deg(x)+\delta_{i}-\delta_{j}
\end{equation}
where $e_{ij}(x)$ is a matrix with $x$ in the $ij$-position and with every
other entry $0$.

Let $R$ be a $\Gamma$-graded ring and $A$, a graded right
$R$-module. Recall that for any $\alpha \in \Gamma$, the $\alpha$-\textit{suspension} of the module $A$ is defined as $A(\alpha
)=\bigoplus_{\gamma\in\Gamma}A(\alpha)_{\gamma}$, where $A(\alpha)_{\gamma
}=A_{\alpha+\gamma}$.

\begin{definition}\label{jjuuii}
\label{fgfhr} Let $R$ be a $\Gamma$-graded ring and $A$, a graded right
$R$-module. We say $A$ is \textit{graded directly-finite} if 
$A\cong_{\gr} A(\alpha) \oplus C$, for some $\alpha\in\Gamma$, then
$C=0$.
\end{definition}

\noindent A ring $R$ with identity $1$ is said to be \textit{directly-finite} if for any
two elements $x,y\in R$, $xy=1$ implies $yx=1$. 
%

A $\Gamma$-graded ring $R$ with identity is called \textit{graded directly-finite} if for any two homogeneous elements $x, y\in R^h$  we have that $xy=1$ implies $yx=1$. Clearly here if $x\in R_{\alpha}$ then $y\in R_{-\alpha}$. 

\begin{proposition} \cite[Proposition 3.2]{HRS-1}
Let $R$ be a $\Gamma$-graded ring and $A$ a graded right
$R$-module. Then $A$ is a graded
directly-finite $R$-module if and only if $\operatorname{END}_{R}(A)$ is a
graded directly-finite ring.
\end{proposition}

In the structure theory of self-injective rings various types of idempotents play important roles. We start with giving the graded version of these definitions. 

\begin{definition}
Let $R=\bigoplus_{\gamma\in\Gamma}R_{\gamma}$ be a $\Gamma$-graded ring.
Denote by $I^{\operatorname{gr}}(R)$ the set of homogeneous idempotents and by
$B^{\operatorname{gr}}(R)$ the set of central homogeneous idempotents.

\begin{enumerate}[\upshape(1)]
\item We say $e\in I^{\operatorname{gr}}(R)$ is \textit{graded abelian
idempotent} if $I^{\operatorname{gr}}(eRe)=B^{\operatorname{gr}}(eRe)$. We say
$R$ is \textit{graded abelian} if $1$ is a graded abelian idempotent, i.e.,
$I^{\operatorname{gr}}(R)=B^{\operatorname{gr}}(R)$.

\medskip

\item We say $e\in I^{\operatorname{gr}}(R)$ is \textit{graded directly-finite} if $eRe$ is a graded directly-finite ring, i.e., if $x,y\in (eRe)^{h}$
such that $xy=e$ then $yx=e$. Equivalently, $eR$ is not graded isomorphic to a proper graded direct summand of $eR$. 

\medskip

\item We say $e\in I^{\operatorname{gr}}(R)$ is \textit{graded faithful} if $y\in
B^{\operatorname{gr}}(R)$ such that $ey=0$, then $y=0$.
\end{enumerate}
\end{definition}

\begin{definition}
Let $R$ be a $\Gamma$-graded von Neumann regular, right self-injective ring.

\begin{enumerate} [\upshape(1)]
\item We say $R$ is of \textit{gr-Type I}, if $R$ contains a graded faithful
abelian idempotent.

\medskip

\item We say $R$ is of \textit{gr-Type II}, if $R$ contains a graded faithful
directly finite idempotent but contains no nonzero graded abelian idempotent.

\medskip

\item We say $R$ is of \textit{gr-Type III}, if $R$ contains no nonzero graded
directly finite idempotent.
\end{enumerate}
\end{definition}

Recall that a $\Gamma$-graded ring $R$ is \emph{strongly graded} if $R_\alpha R_\beta=R_{\alpha + \beta}$, for every $\alpha,\beta \in \Gamma$ (see~\cite{H}).

\begin{lemma}
\label{zerocomponent} If $R$ is a $\Gamma$-graded regular self-injective ring then
$R_{0}$ is a regular self-injective ring. Furthermore if $R$ is strongly
graded, the converse also holds.
\end{lemma}

\begin{proof}
If $R$ is graded regular, clearly $R_{0}$ is regular and consequently $R$ is
flat over $R_{0}$. Let $I$ be a right ideal of $R_{0}$. Tensoring any map $f:I
\rightarrow R_{0}$ with $R$, thanks to $R$ being graded self-injective,
$f\otimes1$ lifts to a graded homomorphism $\overline f:R\rightarrow R$ and 
its zero component $\overline f_{0}\in\operatorname{Hom}(R_{0},R_{0})$
extends $f$. Thus $R_{0}$ is self-injective as well.

If $R$ is strongly graded, by Dade's theorem ${\operatorname{{Gr}^{}-}} R$ is
Morita equivalent to $\operatorname{Mod-} R_{0}$ (\cite[\S 1.5]{H}). Thus any diagram can be
completed via going to $\operatorname{Mod-} R_{0}$.
\end{proof}

\begin{lemma}
If a graded regular self-injective ring $R$ has graded type I (type II), then $R_0$ also has type I (type II).
\end{lemma}

\begin{proof}
Suppose $R$ is a $\Gamma$-graded regular self-injective ring. There is a bijection between the right (left) ideals of $R_0$ and graded right (left) ideals of $R$ which sends any right ideal $A$ of $R_0$ to the graded right ideal $AR$. Moreover $(AR)_0 = A$. From this and \cite[Propositions 10.4 (b) and 10.8 (b)]{G}, we can conclude that if $R$ has graded type I (type II), then $R_0$ also has type I (type II).
\end{proof}

\begin{theorem}
Let $R$ be a $\Gamma$-graded regular ring and $A$ be a graded projective right
$R$-module. Then any graded finitely generated submodule $B$ of $A$ is a
direct summand of $A$. In particular $B$ is graded projective.
\end{theorem}

\begin{proof}
Since $A$ is graded projective, there is a graded module $C$ such that
$A\oplus C\cong_{\operatorname{gr}}F$, where $F \cong_{\operatorname{gr}}
\bigoplus_{i \in I} R(\alpha_{i})$, is a graded free $R$-module. Consider $A$
as a graded submodule of $F$. Since $B\leq^{\operatorname{gr}} A$ is finitely
generated, there is a graded finitely generated free $R$-module $G$ which is a
direct summand of $F$ and contains $B$.

Suppose $B$ is generated by $n$ elements $b_{i}$ of degree $\alpha_{i}$,
$1\leq i \leq n$, respectively, and consider a graded free module $H$
containing $G$ which has a graded basis of at least $n$ elements $h_{i}$ of
degrees $\alpha_{i}$'s. Sending $h_{i}\mapsto b_{i}$, $1\leq i \leq n$, and
the rest of the basis to zero, will define a graded homomorphism of degree
zero $f:H\rightarrow H$ such that $f(H)=B$. Since $\operatorname{End}_{R}(H)$
is graded von Neumann regular, there is $g\in\operatorname{End}_{R}(H)_{0}$
such that $fgf=f$. Thus $fg$ is a homogeneous idempotent and $fg(H)=f(H)=B$.
Thus
\[
H = fg(H)\bigoplus(1-fg)(H)=B\bigoplus(1-fg)(H).
\]
and so $B$ is a direct summand of $H$. (Note that no suspension of $B$ is
required as $f$ and $g$ are of degree zero.) To finish the proof, we use the
following fact twice: if $N\leq M \leq F$ and $N$ and $M$ are direct summand
of $F$, then $N$ is a direct summand of $M$.
\end{proof}

Recall that a graded $R$-module $A$ is called \emph{graded quasi injective}  if for any graded $R$-submodule $M$, and graded homomorphism $f$, there is a graded homomorphism $g$ such that one can complete the following diagram. 
\[
\xymatrix{ M \ar@{^{(}->}[r]^{i} \ar[d]_{f}& A \ar@{.>}[dl]^g\\ A & }
\]

Our next theorem extends \cite[Theorem 1.22]{G} to the graded setting. Since in the following Theorem, the $R$-module $A$ is not necessarily finitely generated, we need to work with the graded ring $\operatorname{END}_R(A)$ (see~(\ref{hhh})) as apposed to the nongraded case which we consider the ring $\End_R(A)$.
Throughout the note, we denote the graded Jacobson radical of a $\Gamma$-graded ring $R$ by $J^{\gr}(R)$.  
Recall also that for $r\in R$, writing $r=\sum_{\alpha\in \Gamma} r_\alpha$, we denote by $\operatorname{supp}(r):=\{r_\alpha \mid r_\alpha \not = 0\}$.

\begin{theorem} \label{ch1mm} 
Let $R$ be a $\Gamma$-graded ring, and $A$ be a graded
quasi-injective right $R$-module. Let $Q=\operatorname{END}_{R}(A)$. Then

\begin{enumerate} [\upshape(a)]
\item $J^{\operatorname{gr}}(Q)=\big \{f\in\operatorname{END}_{R}(A) \mid\ker
f_{\alpha}\leq^{\operatorname{gr}}_{e} A,\text{ for all } \alpha
\in\operatorname{supp}(f)\big \}$.

\medskip

\item $Q/J^{\operatorname{gr}}(Q)$ is graded regular.

\medskip

\item If $J^{\operatorname{gr}}(Q)=0$ then $Q$ is graded right self-injective.
\end{enumerate}
\end{theorem}

\begin{proof} (a) and (b) 
Let \[K=\big \{f\in\operatorname{END}_{R}(A) \mid\ker f_{\alpha}\leq
^{\operatorname{gr}}_{e} A,\text{ for all } \alpha\in\operatorname{supp}
(f)\big \}.\] One checks that $K$ is a graded two sided ideal of $A$. We first
show that $K\subseteq J^{\operatorname{gr}}(Q)$. It is enough to show that
$K^{h} \subseteq J^{\operatorname{gr}}(Q)$. Let $f \in K^{h}$, i.e.,
$f\in\operatorname{Hom}(A,A)_{\alpha}$, for some $\alpha\in\Gamma$ and $\ker f
\leq^{\operatorname{gr}}_{e} A$. We show that for any $r\in\operatorname{Hom}%
(A,A)_{-\alpha}$, $1-rf$ is invertible. We have $\ker(1-rf)\cap\ker f=0$.
Since $\ker f$ is essential, and $\ker(1-rf)$ is a graded right ideal of $A$, it
follows that $\ker(1-rf)=0$. Thus $\theta:=1-rf:A\rightarrow(1-rf)A$ is graded
isomorphism. Since $A$ is graded quasi-injective, there is a graded
homomorphism $g$ which completes the following diagram.
\[
\xymatrix{ (1-rf)A \ar@{^{(}->}[r] \ar[d]_{\theta^{-1}}& A \ar@{.>}[dl]^g\\ A & }
\]
It follows that $g(1-rf)=1$. So $f\in J^{\operatorname{gr}}(Q)$ and
consequently $K\subseteq J^{\operatorname{gr}}(Q)$. Next we show that $Q/K$ is
a graded regular ring. Let $f \in Q^{h}$, i.e., $f\in\operatorname{Hom}%
(A,A)_{\alpha}$ for some $\alpha\in\Gamma$. Set $S=\{N \leq^{\operatorname{gr}%
} A \mid N \cap\ker f =0 \}$. This is a poset in which each chain has an upper
bound. Thus by Zorn's lemma $S$ has a maximal element $B$. It follows that
$B\oplus\ker f \leq^{\operatorname{gr}}_{e} A$. Since $\theta=f: B \rightarrow
f(B)$ is graded isomorphism, and $A$ is graded quasi-injective, there is
graded homomorphism $g$ such that
\[
\xymatrix{ f(B) \ar@{^{(}->}[r] \ar[d]_{\theta^{-1}}& A \ar@{.>}[dl]^g\\ A & }
\]
It follows that $gf=1$ on $B$. Thus $(fgf-f)B=0$. So $B\oplus\ker f \leq
\ker(fgf-f)$ and therefore $fgf-f\in\operatorname{Hom}(A,A)_{\alpha}$ and
$\ker(fgf-f) \leq^{\operatorname{gr}}_{e} A$. So $fgf-f \in K$. Thus in $Q/K$
we have $\overline{fgf}=\overline{f}$, i.e., $Q/K$ is graded regular. So
$J^{\operatorname{gr}}(Q/R)=0$, which gives that $J^{\operatorname{gr}}(Q)=K$.
So $Q/J^{\operatorname{gr}}(Q)$ is graded regular. This gives (a) and (b). 

(c) Suppose $J^{\operatorname{gr}}(Q)=0$. Then $Q$ is graded regular and $A$
is the graded left $Q$-module. So $A$ is (graded) flat over $Q$. Suppose $J$
is a graded right ideal of $Q$ and $f:J\rightarrow Q$ a graded $Q$-module
homomorphism. We will show that there is a $h\in\operatorname{Hom}(A,A)_{0}$
such that one can complete the following diagram and thus $Q$ is
self-injective.
\begin{equation}
\label{jhmu76}\xymatrix{ J \ar@{^{(}->}[r] \ar[d]_{f}& Q \ar@{.>}[dl]^h\\ Q & }
\end{equation}
We have the following commutative diagram of graded maps.
\[
\xymatrix{ J\otimes_Q A \ar[r]^{f\otimes1} \ar[d]_{\cong}& Q\otimes_Q A \ar[d]^{\cong}\\ JA \ar[r]_g & A }
\]
Since $A$ is quasi-injective the graded homomorphism $g$ extends to a
$h\in\operatorname{Hom}(A,A)_{0}$ such that $h(xy)=g(x)y=f(x)y$ for any $x\in
J$ and $y\in A $. Thus we have $f(x)=h(x)$ for any $x\in J$. This completes
Diagram~\ref{jhmu76} and thus the proof.
\end{proof}

\begin{corollary}
\label{selfinj} Let $R$ be a graded ring, $A$ a graded right $R$-module and
$Q=\operatorname{END}_{R}(A)$. If $A$ is graded semi-simple or graded
non-singular quasi-injective, then $Q$ is graded regular self-injective ring.
\end{corollary}

\begin{proof}
We show that $J^{\operatorname{gr}}(Q)=0$ and the corollary follows from
Theorem~\ref{ch1mm}. If $f\in J^{\operatorname{gr}}(Q)$ then $\ker f_{\alpha
}\leq^{\operatorname{gr}}_{e} A$ for all $\alpha\in\operatorname{supp}(f)$. If
$A$ is graded semi-simple, then $f$ is homogeneous and $\ker f=A$ and so $f=0$. If
$A$ is graded non-singular, then $f_{\alpha}:A\rightarrow A(-\alpha)$ is a
graded homomorphism. Thus $A/\ker f \cong A(-\alpha)$. If $A$ is non-singlaur
then $A(-\alpha)$ is non-singular as well and so $f=0$. Thus
$J^{\operatorname{gr}}(Q)=0$.
\end{proof}

The corollary below follows immediately from Corollary~\ref{selfinj} and it gives
a prototype example of graded regular self-injective rings.

\begin{corollary}
Let $R$ be a graded division ring and $A$ be a graded right $R$-module. Then
$\operatorname{END}_{R}(A)$ is a graded regular self-injective ring.
\end{corollary}

\begin{corollary}
Let $R$ be a graded ring such that $R$ is graded right nonsingular. Then its maximal graded right quotient ring $Q^{\operatorname{gr}}(R)$ is graded regular and graded right self-injective.   
\end{corollary}

Since the Leavitt path algebra $L_{K}(E)$ associated to an arbitrary graph $E$ is graded regular, it is graded nonsingular. Consequently, its maximal graded right quotient ring is graded regular and graded right self-injective.

Since for any central homogeneous idempotent $e$, $eR (1-e)R=0$, if $R$ is
graded prime, it follows that $B^{\operatorname{gr}}(R)=\{0,1\}$. If $R$ is,
in addition, graded regular and abelian, i.e., $I^{\operatorname{gr}%
}(R)=B^{\operatorname{gr}}(R)$, then for any $0\not = x\in R^{h}$, there is
$y\in R^{h}$ such that $xyx=x$, so $xy \in B^{\operatorname{gr}}(R)$, and thus
$xy=1$. It follows that $R$ is a graded division ring. The following
proposition shows that if a corner of $R$ is abelian, then $R$ is graded
isomorphic to matrices over a graded division ring.

\begin{theorem} \label{sochgh} 
Let $R$ be a $\Gamma$-graded ring. Then
$R\cong_{\operatorname{gr}}\operatorname{END}_{D}(A)$, where D is a graded
division ring and $A$ is a graded $D$-module if and only if $R$ is graded
prime, graded regular, self-injective ring with $\operatorname{soc}%
^{\operatorname{gr}}(R_{R})\not = 0$.
\end{theorem}

\begin{proof}
Suppose $R\cong_{\operatorname{gr}}\operatorname{END}_{D}(A)$, where $D$ is a
graded division ring and $A$ is a graded $D$-module. Then by
Corollary~\ref{selfinj}, $R$ is graded regular, self-injective. For $0\not =
x,y\in\operatorname{END}_{D}(A)^{h}$, there are homogeneous elements $a,b\in
A$ such that $x(a)\not =0$ and $y(b)\not =0$. Since $y(b)$ is homogeneous, one
can extend $\{y(b)\}$ to a graded basis for $A$~\cite[\S 1.4]{H}. Thus there is a
$z\in\operatorname{END}_{D}(A)^{h}$ such that $z(y(b))=a$. It follows that
$x(z(y(b)))=x(a)\not =0$. Thus $xzy \not =0$. This shows that $R$ is a graded
prime ring. Furthermore, one can choose a homogeneous element $e \in
R=\operatorname{END}_{D}(A)$ such that $eRe\cong_{\operatorname{gr}} D$. Thus
$\operatorname{End}_{R}(eR,eR)\cong_{\operatorname{gr}} eRe\cong%
_{\operatorname{gr}} D$ and the fact that $R$ is graded prime, implies that
$eR$ is graded simple. Thus $\operatorname{soc}^{\operatorname{gr}}(R)\not =0$.

For the converse of the theorem, since $\operatorname{soc}^{\operatorname{gr}%
}(R)\not =0$, one can choose a homogeneous idempotent $e\in R$ such that $eR$
is graded right simple. Thus $eRe \cong_{\operatorname{gr}} \operatorname{End}%
_{R}(eR,eR)$ is a graded division ring. Since $R$ is graded prime,
$B^{\operatorname{gr}}(R)=\{0,1\}$. Thus by the graded version of \cite[Theorem 9.8]{G}, the graded ring
homomorphism $R\rightarrow\operatorname{END}_{eRe}(Re)$ is an isomorphism.
\end{proof}

\begin{proposition} \label{primregselfinjtype1} 
Let $R$ be a $\Gamma$-graded prime,
regular and right self-injective ring. Then $R$ is gr-Type I if and only if
$R$ is graded isomorphic to $\operatorname{END}_{D}(A)$, where $D$ is a graded
division ring and $A$ is a graded right $D$-module.
\end{proposition}

\begin{proof}
Suppose $R\cong_{\operatorname{gr}}\operatorname{END}_{D}(A)$, where D is a
graded division ring and $A$ is a graded $D$-module. Then $\operatorname{soc}%
^{\operatorname{gr}}(R)\not = 0$ by Theorem~\ref{sochgh}. Thus there is a
homogeneous idempotent $e \in R$ such that $eR$ is graded right simple
$R$-module. Thus $\operatorname{End}_{R}(eR,eR) \cong eRe$ is a graded
division ring. Thus $e$ is graded abelian idempotent. Since $R$ is graded
prime, $B^{\operatorname{gr}}(R)=\{0,1\}$. Hence $e$ is faithful as well, and
so $R$ is gr-Type I.

For the converse of the theorem, suppose $R$ is gr-Type I. Thus there is
homogeneous idempotent $e\in R$ which is faithful and abelian. Hence, $eRe$ is
graded prime, regular and abelian. It follows that $R$ is a graded division ring
(see the argument before Theorem~\ref{sochgh}) and consequently, we have that $eR$ is graded
simple. Thus $\operatorname{soc}^{\operatorname{gr}}(R)\not =0$. Now
Theorem~\ref{sochgh} implies that $R\cong_{\operatorname{gr}}%
\operatorname{END}_{D}(A)$.
\end{proof}

\begin{proposition}
\label{ds} Let $R$ be a graded regular self-injective ring and $A$ be a graded right ideal of $R$. Then there is a unique
graded direct summand $B$ of $R$ such that $A\leq_{e}^{\gr}B$.
\end{proposition}

\begin{proof}
Let $B$ be the graded injective envelope of $A$. Then clearly $B$ is a graded direct summand of $R$ and $A\leq_{e}^{\gr}B$. Note that $B/A$ is graded singular and $R/B$ is graded nonsingular and hence $B/A$ equals the graded singular submodule of $R/A$. If $B'$ is any graded direct summand of $R$ with $A\leq_{e}^{\gr}B'$, then $B'/A$ also equals the graded singular submodule of $R/A$. This shows $B=B'$.
\end{proof}

\begin{proposition}\label{whyheadij}
\label{9.5} Let $J$ be a graded two-sided ideal of $R$.

\begin{enumerate}[\upshape(a)]

\item There is a unique homogeneous idempotent $e\in B^{\gr}(R)$ such that
$J_{R}\leq_{e}^{\gr}eR_{R}$.

\medskip

\item If $R/J$ is graded non-singular, then $J=eR$.
\end{enumerate}
\end{proposition}

\begin{proof}
The following proof is essentially Goodearl's proof in \cite{G}
with minor modifications.

(a) By Proposition \ref{ds}, $J\leq_{e}^{\gr}D$, a graded direct summand of
$R_{R}$ which is unique. Now $D=eR$ where $e$ is a homogeneous idempotent of
degree $0$. (Note that $e=\pi(1)$, where $\pi:R\rightarrow D$ is the graded
coordinate projection. So $\pi(1)$ is homogeneous, as $1$ is homogeneous of
degree $0$. Since $\pi(1)$ is also an idempotent, it is of degree $0$). Given
any homogeneous element $x\in R$, we have $xJ\leq^{\gr}J_{R}\leq_{e}^{\gr}eR$
and so $(1-e)xJ=0$. Then the left multiplication by the homogeneous element
$(1-e)x$ induces a graded morphism from the graded singular module $eR/J$ to
$(1-e)xeR$. So $(1-e)xeR$ is a graded singular submodule of $(1-e)R$. As
$(1-e)R$ is graded non-singular, $(1-e)xeR=0$. This holds for all homogeneous
elements $x$ and, consequently, $(1-e)ReR=0$ and so $(1-e)Re=0$. This implies
$(eR(1-e)R)^{2}=0$ and so $(eR(1-e))R=0$ from which we obtain $eR(1-e)=0$. So
the Pierce decomposition of $R$ corresponding to $e$ becomes $R=$ $eRe+$
$(1-e)R(1-e)$. Writing each element $a\in R$ in terms of the last equation, it
is clear that $e$ commutes with $a $ and so $e\in B^{\gr}(R)$ and $J\leq_{e}%
^{\gr}eR$, as desired. If there is another homogeneous idempotent $f\in B(R)$
such that $J\leq_{e}^{\gr}fR$, then by the uniqueness of $D$, $eR=D=fR$ which
implies $e=f$. Hence $e$ is unique.

(b) Now $eR/J$ is graded singular and so if $R/J$ is graded non-singular, then
$eR/J=0$ Hence $J=eR$, where $e\in B^{\gr}(R)$.
\end{proof}

\begin{definition}
\label{fthfulIdempt} A graded right $R$-module $M$ is said to be \textit{graded faithful} if for 
each homogeneous element $r\in R$, there is a homogeneous element $a\in M$
such that $ar\neq0$. If a graded right $R$-module is graded faithful, it is
also faithful.
\end{definition}

We give below the description of graded faithful idempotents.

\begin{lemma}
\label{10.1} Let $R$ be a graded regular self-injective ring. The following properties are equivalent for a homogeneous
idempotent $e\in R$. 

\begin{enumerate}[\upshape(a)]

\item $e$ is graded faithful;

\medskip

\item $(ReR)_{R}$ $\leq_{e}^{\gr}R_{R}$;

\medskip

\item $Re$ is a  graded left faithful ideal;

\medskip

\item $eR$ is a  graded right faithful ideal;

\medskip

\item $\HOM(eR,J)\neq0$ for all non-zero graded right ideals $J$ of $R$.
\end{enumerate}
\end{lemma}

\begin{proof}
The proof similar to the one in \cite{G}, after minor modification
to handle the graded version, works here. For the sake of completeness, we
outline the proof.

Assume (a). By Proposition \ref{9.5}, $(ReR)_{R}\leq_{e}^{\gr}(fR)_{R}$ for
some homogeneous idempotent $f\in B^{\gr}(R)$. Now $(1-f)e=0$ and so $1-f=0$.
Consequently, $(ReR)_{R}\leq_{e}^{\gr}R_{R}$, thus proving (b).

Assume (b). Now $R$ is graded non-singular. Since $(ReR)_{R}$ is graded
essential right ideal, it is then clear that for any homogeneous element $x\in
R$, $x(ReR)\neq0$. In particular, $xRe\neq0$ and it is then clear there is a
homogeneous element $a\in Re$ such that $xa\neq0$.

Assume (c). Let 
\begin{multline}
K=\big \{x\in R \mid  x=x_{i_{1}}+\cdot\cdot\cdot+x_{i_{n}} \text{ is 
the homogeneous decomposition of } x \text{ and } \\ eRx_{i_{k}}=0 \text{ for all }
k=1,\cdot\cdot\cdot,n \big \}.
\end{multline}
 It is easy to see that $K$ is a graded left ideal
of $R$ and $eK=0$. So $(KRe)^{2}=KR(eK)Re=0$. Consequently, $KRe=0$. By (c),
$K=0$. Hence $eR$ is faithful, thus proving (d).

Assume (d). Let $J$ be any graded right ideal and let $0\neq x\in J$ be a
homogeneous element. Since $eR$ is graded faithful, there is a homogeneous
element $y\in eR$ such that $yx\neq0$. Thus $yxR$ is a non-zero graded
projective (principal) right ideal of $R$. Then the graded epimorphism
$xR\rightarrow yxR$ given by $xa\mapsto yxa$ splits with the split map
$\phi:yxr\rightarrow xR$. Since $yxR$ is a graded direct summand, $\phi$
extends to a non-zero graded morphism $\theta:eR\rightarrow J$. Hence
$\HOM(eR,J)\neq0$, thus proving (e).

Assume (e). Suppose $f\in B(R)$ is a homogeneous central idempotent such that
$ef=0$. Then for any graded morphism $\theta:eR\rightarrow fR$, we have
$\theta(ea)=\theta(e)a=fba=bfe=0$. This contradicts (d), as $\HOM(eR,fR)\neq0$ for non-zero $fR$. Hence $f=0$ and we conclude that $e$ is a
graded faithful idempotent. This proves (a).
\end{proof}

\begin{theorem}\label{rgagaga}
\label{typeI} For a graded regular graded right self-injective ring, the following are equivalent.

\begin{enumerate}[\upshape(a)]
\item $R$ has type I;

\medskip

\item Every non-zero graded right ideal of $R$ contains a non-zero graded abelian idempotent;

\medskip

\item The graded two-sided ideal generated by all the graded abelian
idempotents in $R$ is graded essential as a graded right ideal.
\end{enumerate}
\end{theorem}

\begin{proof}
Assume (a), so $R$ contains a non-zero graded faithful abelian idempotent
$e$. Let $J$ be a non-zero graded right ideal of $R$. By Lemma \ref{10.1},
there is a non-zero graded morphism $\theta:eR\rightarrow J$. 

Now $J^{\prime}=\theta(eR)=\theta(e)R\subseteq J$ is a graded principal
right ideal of $R$ and so $J^{\prime}=fR$ for some homogeneous idempotent $f
$. Since $fR$ is graded projective, there is a graded morphism $\phi
:fR\rightarrow eR$ such that $\theta\phi=1_{fR}$. Then $g=\phi\theta\in
\End(eR)=$ $eRe$ and $g=ege$ is a homogeneous idempotent such that
$gR\cong_{\gr}fR$. In particular, $eR=(\ker\theta)\oplus gR$. Consequently,
$fRf\cong_{\gr}gRg\subseteqq$ $eRe$. Thus $f$ is a graded homogeneous
idempotent belonging to $J$. This proves (b). 


\medskip 

Assume (b). Let \[X=\big \{ \text{homogeneous idempotents } e\in B(R) \text{ such that } eR \text{ is
type I}\big \}.\]  Note that $X\subseteq B^{\gr}(R)\cap R_{0}\subseteq B(R_{0})$ and, by
Lemma \ref{zerocomponent} and \cite[Proposition 9.9]{G}, $B(R_{0})$
is a complete lattice. Let $\sup(X)=f\in B(R_{0})$. We claim $f=1$. Suppose
$f\neq1$. Then $1-f\in R_{0}$ is a homogeneous idempotent and, by (b),
$(1-f)R$ contains a homogeneous abelian idempotent $g$. By Proposition
\ref{9.5}, $RgR\leq_{e}^{\gr}eR$ for a unique homogeneous $e\in B^{\gr}(R)$. We
then appeal to Lemma \ref{10.1} to conclude that $g$ is faithful in the ring
$eR$. Hence $eR$ is of type $I$. Thus $e\in X$ and so $e\leq f$, that is
$e=fe$. Hence $g\in fR$. But $g\in(1-f)R$ and so $g=0$, a contradiction. Hence
$f=1$.

Let $Y\subseteq X\subseteq B(R_{0})$ be a maximal family of mutually
orthogonal homogeneous idempotents. Let $\epsilon=\sup(Y)$. Now $\epsilon=1$,
since otherwise $Y\cup\{1-\epsilon\}$ will contradict the maximality of $\ X$.
Write $Y=\{e_{i}:i\in I\}$. Since $e_{i}\in X$, $e_{i}R$ contains a
homogeneous faithful abelian idempotent $h_{i}$. Then $h= \langle\cdot\cdot
\cdot,h_{i},\cdot\cdot\cdot \rangle\in S=%
{\displaystyle\prod\limits_{i\in I}}
e_{i}R$ is an idempotent which is clearly faithful. Also if $\epsilon
=\langle \cdot\cdot\cdot,\epsilon_{i},\cdot\cdot\cdot \rangle$ is a homogeneous idempotent
in $h(%
{\displaystyle\prod\limits_{i\in I}}
e_{i}R)h=$ $%
{\displaystyle\prod\limits_{i\in I}}
h_{i}(e_{i}R)h_{i}$, then since for each $i$, $\epsilon_{i}$ is central in
$h_{i}(e_{i}R)h_{i}$, $\epsilon$ is central in $h(%
{\displaystyle\prod\limits_{i\in I}}
e_{i}R)h$. Thus $h$ is abelian and so $%
{\displaystyle\prod\limits_{i\in I}}
e_{i}R$ is type I. \ Since $\sup\{e_{i}:i\in I\}=1$, we appeal to the graded
version of \cite[Proposition 9.10]{G} to conclude that the
natural ring map $R\rightarrow%
{\displaystyle\prod\limits_{i\in I}}
e_{i}R$ is a graded isomorphism. Hence $R$ is of type I, thus proving (a).

Assume (c). Then the graded two-sided ideal $T$ generated by all the graded abelian idempotents in $R$ is graded essential as a graded right
ideal in $R$. Let $x$ be a non-zero homogeneous element in $R$. Since $R$ is
graded non-singular, $xT\neq0$. This implies that $xre\neq0$ for some graded
abelian idempotent and some homogeneous element $r\in R$. (because, $xre=0$
for all such $e$ and for all homogeneous $r\in R$ will imply that $xa=0$ for
all $a\in T$). Then the graded cyclic right ideal $xreR=fR$ for some
homogeneous idempotent $f\in xR$. Now the map $\theta:ey\mapsto xrey$ is a
graded non-zero morphism from $eR$ to $xR$ with image $xreR=fR$ and since $fR$
is graded projective,  there is a graded morphism $\phi:fR\rightarrow eR$ such that $\theta
\phi=1_{fR}$. Then proceeding as in the proof of (a) $\Rightarrow$  (b), $g=\phi\theta$ is
a homogeneous idempotent in $eRe$ with $gR\cong_{\gr}fR$. So $fRf\cong%
_{\gr}gRg\subseteqq$ $eRe$. This shows that $f\in xR$ is a homogeneous abelian idempotent.  This proves (b).

Now (b) obviously implies (c).
\end{proof}

\begin{theorem}\label{rgagaga1}
Let $R$ be a graded regular graded right self-injective ring. Then
the following properties are equivalent:

\begin{enumerate}[\upshape(a)]

\item $R$ has type II;

\medskip

\item Every non-zero graded right ideal of $R$ contains a non-zero graded
directly finite idempotent;

\medskip

\item The graded two-sided ideal generated by all the graded directly finite
idempotents of $R$ is graded essential as a right ideal of $R$.

\end{enumerate}
\end{theorem}

\begin{proof} 
The proof follows the same line of the proof of Theorem~\ref{rgagaga} by using the
fact that if $e$ is a graded directly finite idempotent, then any homogeneous
idempotent $f\in eRe$ is also graded directly finite.
\end{proof}

\begin{theorem}
\label{mainth} Let $R$ be a $\Gamma$-graded regular graded right
self-injective ring. Then $R \cong_{\operatorname{gr}} R_{1}\times R_{2}
\times R_{3}$, where $R_{1}, R_2, R_3$ are $\Gamma$-graded rings of gr-Type I, II and
III, respectively. Furthermore, this decomposition is unique.
\end{theorem}

\begin{proof}
Let $X$ be the set of all graded abelian idempotents in
$R$. By Proposition~\ref{whyheadij}, there is a unique homogeneous idempotent $e_{1}\in
B^{\gr}(R)$ such that $(RXR)_{R}\leq_{e}^{\gr}e_{1}R$. Theorem~\ref{rgagaga} implies that the
graded regular graded right self-injective ring $e_{1}R$ is of type I. Now the
ring $(1-e_{1})R$ has no non-zero abelian idempotents since \ $(1-e_{1})R\cap
X=\emptyset$. Let $Y$ denote the set of all graded directly finite idempotents
in $(1-e_{1})R$. Again, by Proposition~\ref{whyheadij}, there exists a unique homogeneous
idempotent $e_{2}$ $\in B^{\gr}((1-e_{1})R)$ such that $(RYR)_{R}\leq_{e}^{\gr}e_{2}%
R$. Theorem~\ref{rgagaga1} implies that the ring $e_{2}R$ is of type II. Let
$e_{3}=1-e_{1}-e_{2}$. Since $e_{3}R\cap Y=\emptyset$, the ring $e_{3}R$
contains no non-zero directly finite graded idempotents and consequently, $e_{3}R$ is a graded
regular graded right self-injective ring of type III. Clearly $e_{1}%
,e_{2},e_{3}$ are homogeneous orthogonal central idempotents with $e_{1}%
+e_{2}+e_{3}=1$. Consequently, $R=_{\gr}e_{1}R$ $\oplus e_{2}R\oplus e_{3}R$.

To prove the uniqueness of the decomposition, suppose $f_{1},f_{2},f_{3}$ are
central orthogonal idempotents such that $f_{1}+f_{2}+f_{3}=1$ and that
$f_{1}R,f_{2}R,f_{3}R$ are of type I, II, III respectively. Now $e_{1}%
=e_{1}f_{1}+e_{1}f_{2}+e_{1}f_{3}$. Here $e_{1}f_{2}\in e_{1}Re_{1}\cap
f_{2}Rf_{2}=0$ and $e_{1}f_{3}\in e_{1}Re_{1}\cap f_{3}Rf_{3}=0$.
Consequently, $e_{1}=e_{1}f_{1}\in f_{1}R$ and so $e_{1}R\subseteq f_{1}R$.
Similarly, $f_{1}R\subseteq e_{1}R$ and so $e_{1}R=f_{1}R$.
In the same way, $e_{2}R=f_{2}R$ and $e_{3}R=f_{3}R$. This finishes the proof. 
\end{proof}

\section{Applications to Leavitt path algebras}

 In this section we apply the results from the previous section to the case of Leavitt path algebras. We first recall the basics of Leavitt path algebras. Let $K$ be a field and $E$ be an arbitrary directed graph. Let $E^0$ be the set of vertices, and $E^1$ be the set of edges of directed graph $E$. Consider two maps $r: E^1 \rightarrow E^0$ and $s:E^1 \rightarrow E^0$. For any edge $e$ in $E^1$, $s(e)$ is called the {\it source} of $e$ and $r(e)$ is called the {\it range} of $e$. If $e$ is an edge starting from vertex $v$ and pointing toward vertex $w$, then we imagine an edge starting from vertex $w$ and pointing toward vertex $v$ and call it the {\it ghost edge} of $e$ and denote it by $e^*$. We denote by $(E^1)^*$, the set of all ghost edges of directed graph $E$. If $v \in E^0$ does not emit any edges, i.e. $s^{-1}(v) = \emptyset$, then $v$ is called a {\it sink} and if $v$ emits an infinite number of edges, i.e. $|s^{-1}(v)| = \infty$, then $v$ is called an {\it infinite emitter}. If a vertex $v$ is neither a sink nor an infinite emitter, then $v$ is called a {\it regular vertex}. 

\bigskip

\noindent The {\it Leavitt path algebra} of $E$ with coefficients in $K$, denoted by $L_K(E)$, is the $K$-algebra generated by the sets $E^0$, $E^1$, and $(E^1)^*$, subject to the following conditions:
	\begin{enumerate}
		\item[(A1)] $v_iv_j = \delta_{ij} v_i$ for all $v_i, v_j \in E^0$.
		\item[(A2)] $s(e)e = e = er(e)$ and $r(e)e^* = e^* = e^*s(e)$ for all $e$ in $E^1$.
		\item[(CK1)] $e_i^*e_j = \delta_{ij} r(e_i)$ for all $e_i, e_j \in E^1$.
		\item[(CK2)]  If $v \in E^0$ is any regular vertex, then $v = \sum_{\{e \in E^1: s(e) = v\}} ee^*$.
	\end{enumerate}
Conditions (CK1) and (CK2) are known as the {\it Cuntz-Krieger relations}. If $E^0$ is finite, then $\sum\limits_{v_i \in E^0} v_i$ is an identity for $L_K(E)$ and if $E^0$ is infinite, then $E^0$ generates a set of local units for $L_K(E)$. 	We refer the reader to \cite{AAS-1} for the basics on the theory of Leavitt path algebras. 

It is known that the Leavitt path algebra $L_K(E)$ over any graph $E$ is a graded regular ring~\cite{H-2}. As an application of the structure theory of graded regular graded self-injective rings, we have the following in case of unital Leavitt path algebras.

\begin{theorem}\label{vonmimi}
Let $L_{K}(E)$ be a Leavitt path algebra associated to a finite graph $E$,
where $K$ is a field. Then the following are equivalent.

\begin{enumerate}[\upshape(a)]

\item $L_{K}(E)$ is an algebra of gr-Type I;

\medskip 

\item $L_{K}(E)$ is graded self-injective;

\medskip

\item No cycle in $E$ has an exit;

\medskip

\item $L_K(E)$ is a graded $\Sigma$-$V$ ring;

\medskip 

\item There is a graded isomorphism
\[
L_{K}(E)\cong_{\operatorname{gr}} {\displaystyle\bigoplus\limits_{v_{i}\in X}}
\Ma_{\Lambda_{i}}(K)((|\overline{p^{v_{i}}}|)\oplus{\displaystyle\bigoplus
\limits_{w_{j}\in Y}} \Ma_{\Upsilon_{j}}(K[x^{t_{j}},x^{-t_{j}}%
])(|\overline{{q^{w_{j}}}}|)
\]
where $\Lambda_{i}$,$\Upsilon_{j}$ are suitable index sets, the $t_{j}$ are
positive integers, $X$ is the set of sinks in $E$ and $Y$ is the set of all distinct cycles
(without exits) in $E$.

\end{enumerate}
\end{theorem}

\begin{proof}
The equivalence of (b), (c), (d) and (e) is established in \cite[Theorem 4.16]{HRS-1}. Now the equivalence of (a) with these conditions follows from Proposition~\ref{primregselfinjtype1}. 
\end{proof}

In fact the above theorem shows that self-injective unital Leavitt path algebras are of gr-Type $I_f$, i.e., they are of graded Type I and are graded directly finite (see \cite{HRS-1}).

\end{document}